\documentclass[12pt]{amsart}
\usepackage{graphicx}
\usepackage{amssymb}
\usepackage{amsmath}
\usepackage{amscd}
\usepackage{upgreek}
\usepackage{mathrsfs}
\usepackage{stmaryrd}
\usepackage{longtable}
\usepackage{txfonts}
\usepackage[T1]{fontenc}
\usepackage{eucal}
\usepackage{arydshln}
\usepackage{wrapfig}
\usepackage{float}
\usepackage{xypic}
\usepackage{dsfont}
\usepackage{xcolor}
\usepackage{color}
\usepackage[colorlinks,linkcolor=blue,anchorcolor=blue,
citecolor=blue]{hyperref}
\usepackage{hyperref}
\usepackage[hmarginratio=1:1,top=32mm,columnsep=20pt]{geometry}
\usepackage{amssymb,color}
\definecolor{MyGreen}{RGB}{29,162,55} 
	
\allowdisplaybreaks	

\newtheorem{theorem}{Theorem}[section]
\newtheorem{prop}[theorem]{\rm \textsc{Proposition}}

\newtheorem{lem}[theorem]{\rm \textsc{Lemma}}
\newtheorem{coro}[theorem]{\rm \textsc{Corollary}}
\newtheorem{rem}[theorem]{\it Remark}
\newtheorem{conj}[theorem]{\rm \textsc{Conjecture}}

\newtheorem{thm}[theorem]{\it \textsc{Theorem}}

\newcommand{\N}{\mathbb{N}} 
\newcommand{\F}{\mathbb{F}} 
\newcommand{\KK}{\mathbb{K}} 

\newcommand{\ra}{\longrightarrow} 
  
\newcommand{\rank}{{\rm rank}} 
\newcommand{\GL}{{\rm GL}} 
\newcommand{\SL}{{\rm SL}} 
\newcommand{\U}{{\rm U}} 
\newcommand{\Tr}{{\rm Tr}} 
\newcommand{\Qt}{{\rm Quot}} 
\newcommand{\HH}{\mathcal{H}}
\newcommand{\Frob}{{\rm Frob}_p}

\newcommand{\Sb}[2]{{\overline{S}^{#1}_{#2}}}

\newcommand{\gam}{\gammaup} 
 
\newcommand{\lam}{\lambdaup} 
 
\newcommand{\sig}{\sigmaup}

\newcommand{\binomial}[2]{\genfrac{(}{)}{0pt}{}{#1}{#2}}

\begin{document}
\setlength{\oddsidemargin}{0cm}
\setlength{\evensidemargin}{0cm}

\title{\scshape Modular invariants of a vector and a covector: a proof of a conjecture of Bonnaf\'e and Kemper}

\author{Yin Chen}

\address{School of Mathematics and Statistics, Northeast Normal University,
 Changchun 130024, P.R. China}

\email{ychen@nenu.edu.cn}

\author{David L. Wehlau}

\address{Department of Mathematics and Computer Science, Royal Military College, Kingston, ON, K7K 5L0, Canada}

\email{wehlau@rmc.ca}

\date{\today}

\def\shorttitle{Modular invariants of a vector and a covector}

\begin{abstract}
Consider a finite dimensional vector space $V$ over a finite field $\F_q$.
We give a minimal generating set for the ring of invariants $\F_q[V \oplus V^*]^{\GL(V)}$,
and show that this ring is a Gorenstein ring but is not a complete intersection.
These results confirm a conjecture of Bonnaf\'e and Kemper \cite[Conjecture 3.1]{BK2011}.
\end{abstract}



\maketitle
\baselineskip=18pt


\setcounter{subsection}{0}
\renewcommand{\thesubsection}
{\arabic{subsection}}

\setcounter{equation}{0}
\renewcommand{\theequation}
{\arabic{equation}}
\setcounter{theorem}{0}
\renewcommand{\thetheorem}
{\arabic{theorem}}

\subsection{Introduction}

Suppose a group $G$ acts linearly on a $k$-vector space $W$.  This induces an action of $G$ on the dual space $W^{*}$
  by $\sig\cdot l = l\circ \sig^{-1}$ for $\sig \in G$ and $l \in W^*$.  Extending the action on $W^*$
  multiplicatively yields an action of $G$ on the polynomial ring $k[W]$.
  Here $k[W]$ denotes the symmetric algebra on $W^*$.  If $\{x_1,x_2\dots,x_n\}$ is a basis of $W^*$, we may
  write $k[W]=k[x_1,x_2,\dots,x_n]$.  The ring of invariants is the subalgebra
  $k[W]^G := \{f \in k[W] \mid \sig\cdot f = f, \text{ for all } \sig \in G\}$.

In this paper we prove the following result, which was conjectured by Bonnaf\'e and Kemper \cite[Conjecture 3.1]{BK2011}.

\begin{thm}\label{mt}
Let $\F_q$ denote the finite field of characteristic $p$ and having $q$ elements.
Let $n\geq 2$ and let $V$ be an $n$ dimensional vector space over $\F_q$.  Then
$$
\F_q[V\oplus V^*]^{\GL(V)}=\F_q\Big[c_{n,0},c_{n,1},\dots,c_{n,n-1},c_{n,0}^*,c_{n,1}^{*},\dots,c_{n,n-1}^*,u_{1-n},\dots,u_{-1},u_{0},u_{1},\dots,u_{n-1}\Big].$$
Moreover, $\F_q[V\oplus V^*]^{\GL(V)}$ is Gorenstein but is not a complete intersection.
\end{thm}

Here $c_{n,0}, c_{n,1},\dots,c_{n,n-1}$ are the Dickson invariants and $\F_q[V]^{\GL(V)} = \F_q[c_{n,0}, c_{n,1},\dots,c_{n,n-1}]$
which was proved by Dickson \cite[Theorem]{Dic1911}.   See Wilkerson \cite[Theorem 1.2]{Wil1983} for a modern treatment.
Similarly, $c^*_{n,0}, c^*_{n,1},\dots,c^*_{n,n-1}$  are the Dickson invariants in the dual set of variables so that
$\F_q[V^*]^{\GL(V)} = \F_q[c^*_{n,0}, c^*_{n,1},\dots,c^{*}_{n,n-1}]$.  The invariants $u_j$ are defined in Section~\ref{BKth} below.

\begin{rem}{\rm
(1) The analogue of Theorem \ref{mt} for $n=1$ is treated in \cite{BK2011} since $\GL(V)$ is its own its Borel subgroup, $\F_q^*$.
When $n=1$ we have  $\F_q[V\oplus V^*]^{\GL(V)} = \F_q[x_1,y_1]^{\F_q^*} = \F_q[x_{1}^{q-1},x_{1}y_{1},y_{1}^{q-1}]$ is a hypersurface ring.
(2) The special case $n=2$ of Theorem \ref{mt} was proved in Chen \cite{Che2014}.
}
\end{rem}

\subsection{Dickson invariants and Mui invariants}

For simplicity, in the rest of this paper,  we write $\F=\F_{q}$, $G=\GL(V)$, and $U=\U(V)$ for the unipotent group of upper triangular matrices with 1's on the diagonal. 
 We let $\{y_1,y_2,\dots,y_n\}$ denote the standard basis of $V$ and $\{x_1,x_2,\dots,x_n\}$ the dual basis of $V^*$.
 Then $\F[V]=\F[x_1,x_2,\dots,x_n]$, $\F[V^{*}]=\F[y_1,y_2,\dots,y_n]$, and
 $\F[V \oplus V^*] = \F[x_1,x_2,\dots,x_n,y_1,y_2,\dots,y_n]$. 
 
We give two definitions of the Dickson invariants $c_{n,0},c_{n,1},\dots,c_{n,n-1}$. The first definition illustrates 
the importance of Dickson invariants. 
 \begin{equation*}
\prod_{x \in V^*} (\lam-x) = \sum_{i=0}^{n} (-1)^{n-i} c_{n,i} \cdot\lam^{q^i}
\end{equation*}
where $\lam$ is an indeterminate and we make the convention that $c_{n,n}:=1$.

The second definition expresses the Dickson invariants in terms of certain determinants and will
allow us later to derive some relations.
 Consider the following $n\times (n+1)$-matrix whose entires taken from $\F[V]$:
$$D=\begin{pmatrix}
x_{1}&  x_{1}^{q}&\cdots&x_{1}^{q^{n}}  \\
x_{2}&  x_{2}^{q}&\cdots&x_{2}^{q^{n}} \\
\vdots&\vdots&\cdots&\vdots\\
x_{n}&  x_{n}^{q}&\cdots&x_{n}^{q^{n}}
\end{pmatrix}.
$$
For $0\leqslant i\leqslant n$, we let $d_{n,i}$ denote the determinant of the matrix obtained by deleting the $(i+1)$-th column from $D$.
Then
\begin{equation*}
c_{n,0}=\frac{d_{n,0}}{d_{n,n}}=d_{n,n}^{q-1}, \quad c_{n,1}=\frac{d_{n,1}}{d_{n,n}},\quad\dots,\quad c_{n,n-1}=\frac{d_{n,n-1}}{d_{n,n}}.
\end{equation*}

In 1975,  Mui \cite{Mui1975} proved that the ring of invariants $\F[V]^{U}$ of the unipotent group 
 is a polynomial algebra over $\F$, generated by $\{f_{1},f_{2},\dots,f_{n}\}$,
where $f_{1}=x_{1}$, and for $2\leqslant i\leqslant n$,
\begin{equation*}
f_{i}:=\prod_{v\in V_{i-1}}(x_{i}+v)
\end{equation*}
 where $V_{i-1}$ denotes the vector subspace generated by $\{x_{1},x_{2},\dots,x_{i-1}\}$.

\subsection{Bonnaf\'e-Kemper's Theorem}\label{BKth}

In 2011, Bonnaf\'e-Kemper \cite{BK2011} considered the actions of $U$ and $G$ on 
 $\F[V\oplus V^{*}]$, and the corresponding two rings of invariants. 

The ring $\F[V\oplus V^{*}]$ is equipped with an involution 
$$*:\F[V\oplus V^{*}]\ra\F[V\oplus V^{*}], ~~f\mapsto f^{*}$$  given by 
$x_{1}\mapsto y_{n}, x_{2}\mapsto y_{n-1},\dots,x_{n-1}\mapsto y_{2},x_{n}\mapsto y_{1}.$
Clearly 
\begin{eqnarray*}
\F[V^{*}]^{G}&=&\F[c_{n,0}^{*},c_{n,1}^{*},\dots,c_{n,n-1}^{*}] \\
\F[V^{*}]^{U}&=&\F[f_{1}^{*},f_{2}^{*},\dots,f_{n}^{*}].
\end{eqnarray*}

The map $F:\F[V\oplus V^{*}]\ra\F[V\oplus V^{*}]$ is defined by $x_{i}\mapsto x_{i}^{q}, y_{i}\mapsto y_{i}$. Dually,  another one 
$F^{*}:\F[V\oplus V^{*}]\ra\F[V\oplus V^{*}]$ is defined by $x_{i}\mapsto x_{i}, y_{i}\mapsto y_{i}^{q}$.  These two maps commute with
the action of $G$ and so restrict to endomorphisms of $\F[V\oplus V^{*}]^U$ and $\F[V\oplus V^{*}]^G$. 

The natural pairing of $V$ with $V^*$ corresponds to a natural quadratic $G$-invariant in $\F[V\oplus V^{*}]$:
\begin{equation*}
u_{0}:=x_{1}y_{1}+x_{2}y_{2}+\cdots+x_{n}y_{n}.
\end{equation*}
For $i\in\N^{+}$, we define 
\begin{eqnarray*}
u_{i} & := & F^{i}(u_{0}) = x_{1}^{q^{i}}y_{1}+x_{2}^{q^{i}}y_{2}+\cdots+x_{n}^{q^{i}}y_{n}\\
u_{-i} & := & (F^{*})^{i}(u_{0})=x_{1}y_{1}^{q^{i}}+x_{2}y_{2}^{q^{i}}+\cdots+x_{n}y_{n}^{q^{i}}
\end{eqnarray*}
which are $G$-invariants, since $F$ and $F^{*}$ commute with the action of $\GL(V)$. We observe that 
$u_{-i}^{*}=u_{i}$ for all $i\in\N$.

Bonnaf\'e-Kemper \cite[Theorem 2.4]{BK2011} proved that 
\begin{equation}
\label{BKt}
\F[V\oplus V^{*}]^{U}=\F[f_{1},\dots,f_{n},f_{1}^{*},\dots,f_{n}^{*},u_{2-n},\dots,u_{0},\dots,u_{n-2}].
\end{equation}
Furthermore they proved that this ring is a complete intersection and exhibited the minimial relations
among the generating invariants.

\subsection{The Reynolds operator and Campbell-Hughes' Theorem}

  Suppose that a group $G_{1}$ acts linearly on a $k$-vector space $V$.  Let $G_{2}$ be any subgroup of  $G_{1}$.  
  The {\it relative trace} (or {\it relative transfer})
  map is defined by
  $$\Tr_{G_{2}}^{G_{1}} : k[V]^{G_{2}} \ra k[V]^{G_{1}}, \quad g \mapsto \sum_{\sig \in G_{1}/G_{2}} \sig \cdot g$$
  where $G_{1}/G_{2}$ denotes any set of left coset representatives of $G_{2}$ in $G_{1}$.
  It is easy to see that $\Tr_{G_{2}}^{G_{1}}$ is a degree-preserving $k[V]^{G_{1}}$-module homomorphism.
   If in addition, the index $[G_{1}:G_{2}]$ is invertible in $k$, then we have the so-called \textit{Reynolds operator}
   \begin{equation*}
R_{G_{2}}^{G_{1}} = \frac{1}{[G_{1}:G_{2}]} \Tr_{G_{2}}^{G_{1}}
\end{equation*}
 which is a projection from $k[V]^{G_{2}}$ onto $k[V]^{G_{1}}$.

To prove Theorem \ref{mt}, we first note that the group $U$ is a $p$-Sylow subgroup of $G$, so the index $[G:U]$ is invertible in $\F$.
Then the Reynolds operator $R_{U}^{G}: \F[V\oplus V^{*}]^{U} \ra \F[V\oplus V^{*}]^{G}$ is a surjective $\F[V\oplus V^{*}]^{G}$-module homomorphism, which together with a generating set of $\F[V\oplus V^{*}]^{U}$ we have known in Bonnaf\'e-Kemper's result (\ref{BKt}), leads us to use  the Reynolds operator $R_{U}^{G}$ to find a generating set  of $\F[V\oplus V^{*}]^{G}$.

In  \cite{CH1996} Campbell and Hughes studied connections between the Dickson invariants and the Mui invariants.  They constructed
 a set of left coset representatives for $U_{n}(\F_{p})\subset \GL_{n}(\F_{p})$, over the prime field $\F_{p}$.  
They also found a free basis for the Mui invariants $\F_{p}[V]^{U_{n}(\F_{p})}$ as a module over the Dickson invariants 
$\F_{p}[V]^{\GL_{n}(\F_{p})}$.
Although they stated their results working over $\F_p$, their proofs are valid over any finite field. 
In particular, the proof of Campbell-Hughes \cite[Theorem 6.3]{CH1996} shows that 
\begin{equation}
\label{CHt}
\F[V]^{U}=\bigoplus_{\gam\in\Gamma}\F[V]^{G}\cdot \gam
\end{equation}
is a free $\F[V]^{G}$-module, where
\begin{equation*}
\Gamma:=\Big\{f_{1}^{a_{1}}\cdot f_{2}^{a_{2}}\cdots f_{n}^{a_{n}}~\big|~ 0\leqslant a_{1}<q^{n}-1, 0\leqslant a_{2}<q^{n-1}-1,\dots,0\leqslant a_{n}<q-1\Big\}.
\end{equation*}
Of course the analagous result holds for $\F[V^{*}]^{U}$ and $\F[V^{*}]^{G}$. 

\subsection{Sketch of the proof of Theorem \ref{mt}}

We define 
\begin{eqnarray*}
A &:= &\F[ c_{n,0},c_{n,1},\dots,c_{n,n-1},c_{n,0}^{*},c_{n,1}^{*},\dots,c_{n,n-1}^{*},u_{1-n},u_{2-n},\dots,u_{0},\dots,u_{n-2},u_{n-1}]\\
\Omega&:=&\Big\{
f_{1}^{a_{1}}f_2^{a_2}\cdots f_{n}^{a_{n}}\cdot f_{1}^{*b_{1}}f_2^{*b_2}\cdots f_{n}^{*b_n}~\big|~
   0 \leq a_i, b_i \leq q^{n+1-i}-2 \text{ for } i=1,2,\dots,n
 \Big\}.
\end{eqnarray*}

In this section we outline the steps we will follow in order to prove
\begin{equation}\label{first assertion}
 \F[V\oplus V^{*}]^{G}=A
 \end{equation}
  which is the first assertion of Theorem~\ref{mt}.

Since $A \supset \F[V]^G$, combining Bonnaf\'e-Kemper's result (\ref{BKt}), with Campbell-Hughes' result (\ref{CHt}),
we see that $\Omega$ is a generating set for $\F[W]^{U}$ as an $A$-module, i.e.,
\begin{equation}
\label{Amodule}
\F[V\oplus V^{*}]^{U}=\sum_{\omega\in\Omega}A\cdot \omega.
\end{equation}

Since $R_U^G( \F[V \oplus V^*]^U) = \F[V \oplus V^*]^G$ and since $A \subseteq \F[V \oplus V^*]^G$, we see that 
to show (\ref{first assertion}) it suffices to show that  
\begin{equation} \label{goal}
R_{U}^{G}(\omega)\in A \quad\text{ for all }\omega\in\Omega.
\end{equation}

Define 
\begin{eqnarray*}
h&:=&R_{U}^{G}(\omega) \text{ where }\omega\in\Omega\\
B_{k}&:=&\F[c_{n,0},c_{n,1},\dots,c_{n,n-1}, u_{n-1},u_{n-2},\dots,u_{0},u_{-1},\dots,u_{1-n-k}],\quad k=0,1,2,\dots\\
B_{\infty}&:=&\cup_{k=0}^{\infty} B_{k}.
\end{eqnarray*}

Our proof will be separated into the following steps:

\begin{enumerate}
  \item Show that the invariant field $\F(V\oplus V^{*})^{G}=\F(c_{n,0},c_{n,1},\dots,c_{n,n-1},u_{0},\dots,u_{n-1})$.
  \item Show $B_{\infty}\subseteq A$ and $\Qt(B_{\infty})=\Qt(A)=\F(V\oplus V^{*})^{G}$.
  \item Show $\F[V\oplus V^{*}]^{G}[c_{n,0}^{-1}]=\F[c_{n,0},c_{n,1},\dots,c_{n,n-1},u_{0},\dots,u_{n-1}][c_{n,0}^{-1}]$.
  \item Show $B_{\infty}$ is a unique factorization domain, and thus it is a completely integrally closed.
  \item Show that $h$ is a quasi-almost integral element over $B_{\infty}$ which implies that $h\in B_{\infty}\subseteq A$.
\end{enumerate}

\subsection{Some relations in $\F[V\oplus V^{*}]^{G}$}

As we noted above  $c_{n,n}=1=c_{n,n}^{*}$.

\begin{lem}\label{ini}
In $\F[V\oplus V^{*}]^{G}$, we have the following relation:
\begin{equation}
\tag{$T_0$}\label{T0}
c_{n,0}u_{0}-c_{n,1}u_{1}+c_{n,2}u_{2}-\cdots+(-1)^{n-1}c_{n,n-1}u_{n-1}+(-1)^{n}c_{n,n}u_{n}=0.
\end{equation}
\end{lem}

\begin{proof}
We recall the relations $(R_n)$ and $(R_n^+) := F(R_n)$ from Bonnaf\'e-Kemper \cite[page 105]{BK2011}:
    \begin{align}
      \tag{$R_n$}\label{Rn}  \sum_{i=0}^{n-1}(-1)^{i+n+1} c_{n-1,i}\cdot u_i - f_n \cdot f_1^{*}&=0 \\
      \tag{$R_n^+$}  \sum_{i=0}^{n-1} (-1)^{i+n+1} c_{n-1,i}^q\cdot u_{i+1} - f_n^q \cdot f_1^*&=0. \label{Rn+}
    \end{align}
 We will use the equation $c_{n,i} = c_{n-1,i-1}^q + f_n^{q-1} \cdot c_{n-1,i}$ (see Bonnaf\'e-Kemper \cite[page 104]{BK2011} or Wilkerson \cite[Proposition 1.3 (b)]{Wil1983})
     which is valid for $0 \leq i \leq n$.
     Consider
    $f_n^{q-1} \cdot$ (\ref{Rn}) - (\ref{Rn+}).  This yields the equation
    \begin{align*}
      0 &= \sum_{i=0}^{n-1} (-1)^{i+n+1} f_n^{q-1} \cdot c_{n-1,i}  \cdot u_i  - \sum_{i=1}^{n} (-1)^{i+n} c_{n-1,i-1}^q \cdot u_{i}\\
         &=  (-1)^{n+1}f_n^{q-1}c_{n-1,0}\cdot u_0 + \sum_{i=1}^{n-1} (-1)^{i+n+1}\left( f_n^{q-1}\cdot c_{n-1,i} \cdot u_i  + c_{n-1,i-1}^q \cdot u_{i}\right) - (-1)^{2n}c_{n-1,n-1}^q \cdot u_n\\
         & = \sum_{i=0}^n (-1)^{i+n+1}\left( f_n^{q-1} \cdot c_{n-1,i}  + c_{n-1,i-1}^q \right)u_i \\
         &= \sum_{i=0}^n (-1)^{i+n+1} c_{n,i} \cdot u_i=(-1)^{n+1}\left(\sum_{i=0}^n (-1)^{i} c_{n,i} \cdot u_i\right)
    \end{align*}
    which is the relation (\ref{T0}) as desired.
\end{proof}

Applying the map $F^{*}$ on (\ref{T0}) repeatedly, we obtain more relations:
\begin{align*}
 \tag{${T_{1}}$} \label{T1} c_{n,0}u_{-1}-c_{n,1}u_{0}^{q}+c_{n,2}u_{1}^{q}-\dots+(-1)^{n-1}c_{n,n-1}u_{n-2}^{q}+(-1)^{n}u_{n-1}^{q}    & =0  \\
\tag{${T_{2}}$} \label{T2}
c_{n,0}u_{-2}-c_{n,1}u_{-1}^{q}+c_{n,2}u_{0}^{q^{2}}-\dots+(-1)^{n-1}c_{n,n-1}u_{n-3}^{q^{2}}+(-1)^{n}u_{n-2}^{q^{2}} &=  0\\ 
&\vdots\\
\tag{${T_{n-1}}$} \label{Tn-1}
c_{n,0}u_{1-n}-c_{n,1}u_{2-n}^{q}+c_{n,2}u_{3-n}^{q^{2}}-\dots+(-1)^{n-1}c_{n,n-1}u_{0}^{q^{n-1}}+(-1)^{n}u_{1}^{q^{n-1}}&=0.
\end{align*}
Continuing we obtain
\begin{equation}
\tag{${T_j}$} \label{Tj} \sum_{i=0}^{n} (-1)^i c_{n,i} u_{i-j}^{q^{\min(i,j)}} = 0
\end{equation}
for all $j \geq 0$.

We apply the involution $*$ on (\ref{T0}), (\ref{T1}), $\dots$, (\ref{Tn-1}) respectively, and obtain 
\begin{align*}
\tag{${T_{0}^{*}}$}\label{T0s}
c_{n,0}^{*}u_{0}-c_{n,1}^{*}u_{-1}+c_{n,2}^{*}u_{-2}-\cdots+(-1)^{n-1}c_{n,n-1}^{*}u_{1-n}+(-1)^{n}u_{-n}&=0\\
 \tag{$T_{1}^{*}$} \label{T1s} c_{n,0}^{*}u_{1}-c_{n,1}^{*}u_{0}^{q}+c_{n,2}^{*}u_{-1}^{q}-\dots+(-1)^{n-1}c_{n,n-1}^{*}u_{2-n}^{q}+(-1)^{n}u_{1-n}^{q}    & =0  \\
\tag{${T_{2}^{*}}$} \label{T2s}
c_{n,0}^{*}u_{2}-c_{n,1}^{*}u_{1}^{q}+c_{n,2}^{*}u_{0}^{q^{2}}-\dots+(-1)^{n-1}c_{n,n-1}^{*}u_{3-n}^{q^{2}}+(-1)^{n}u_{2-n}^{q^{2}} &=  0\\
&\vdots\\
\tag{$T_{n-1}^{*}$} \label{Tn-1s}
c_{n,0}^{*}u_{n-1}-c_{n,1}^{*}u_{n-2}^{q}+c_{n,2}^{*}u_{n-3}^{q^{2}}-\dots+(-1)^{n-1}c_{n,n-1}^{*}u_{0}^{q^{n-1}}+(-1)^{n}u_{-1}^{q^{n-1}}&=0.
\end{align*}

Applying $*$ to (\ref{Tj}) yields 
\begin{equation}
\tag{${T^*_j}$} \label{Tjs} \sum_{i=0}^{n} (-1)^i c^*_{n,i} u_{j-i}^{q^{\min(i,j)}} = 0
\end{equation}
for all $j \geq 0$.

From the expression for the Dickson invariants in terms of determinants we have
$c_{n,0}\cdot c_{n,0}^{*}=(d_{n,n}\cdot d_{n,n}^{*})^{q-1}$, where
\begin{eqnarray*}
d_{n,n}\cdot d_{n,n}^{*}& = & \textrm{det}\begin{pmatrix}
     x_{1} &x_{2}&\cdots&x_{n}    \\
     x_{1}^{q} &x_{2}^{q}&\cdots&x_{n}^{q}   \\
     \vdots&\vdots&\vdots&\vdots\\
     x_{1}^{q^{n-1}} &x_{2}^{q^{n-1}}&\cdots&x_{n}^{q^{n-1}}   \\
\end{pmatrix}
\textrm{det}\begin{pmatrix}
     y_{1} &y_{1}^{q}&\cdots&y_{1}^{q^{n-1}}    \\
     y_{2} &y_{2}^{q}&\cdots&y_{2}^{q^{n-1}}    \\
     \vdots&\vdots&\vdots&\vdots\\
    y_{n} &y_{n}^{q}&\cdots&y_{n}^{q^{n-1}}    \\
\end{pmatrix}.
\end{eqnarray*}
Thus we have another relation
\begin{equation}
\tag{$T_{00}$}\label{T00}
c_{n,0}\cdot c_{n,0}^{*}-\textrm{det}\begin{pmatrix}
     u_{0} &   u_{-1}&u_{-2}&\cdots&u_{1-n} \\
      u_{1}&  u_{0}^{q}&u_{-1}^{q}&\cdots& u_{2-n}^{q}\\
      u_{2}&u_{1}^{q}&\ddots&\ddots&\vdots\\
      \vdots&\ddots&\ddots&\ddots&u_{-1}^{q^{n-2}}\\
      u_{n-1}&u_{n-2}^{q}&\cdots&u_{1}^{q^{n-2}}&u_{0}^{q^{n-1}}
\end{pmatrix}^{q-1}=0.
\end{equation}

\subsection{The invariant field $\F(V\oplus V^{*})^{G}$} 

In this subsection we study the rationality problem of the invariant field $\F(V\oplus V^{*})^{G}$.

\begin{lem} \label{algindep}
The set
$\{u_{n-1},\dots,u_{1},u_{0},u_{-1},\dots,u_{-n}\}$ is algebraically independent over $\F.$
\end{lem}

\begin{proof}
This follows from the Jacobian criterion (see Benson \cite[Proposition 5.4.2]{Ben1993}), since
$$\textrm{det}\begin{pmatrix}
    \frac{\partial u_{n-1}}{\partial y_{1}} &  \cdots &     \frac{\partial u_{n-1}}{\partial y_{n}} &        \frac{\partial u_{n-1}}{\partial x_{1}} &  \cdots &     \frac{\partial u_{n-1}}{\partial x_{n}}    \\
      \vdots& \cdots &\vdots&\vdots&\cdots&\vdots\\
  \frac{\partial u_{0}}{\partial y_{1}} &  \cdots &     \frac{\partial u_{0}}{\partial y_{n}} &        \frac{\partial u_{0}}{\partial x_{1}} &  \cdots &     \frac{\partial u_{0}}{\partial x_{n}}    \\    
        \vdots& \cdots &\vdots&\vdots&\cdots&\vdots\\
  \frac{\partial u_{-n}}{\partial y_{1}} &  \cdots &     \frac{\partial u_{-n}}{\partial y_{n}} &        \frac{\partial u_{-n}}{\partial x_{1}} &  \cdots &     \frac{\partial u_{-n}}{\partial x_{n}} 
\end{pmatrix}=\textrm{det}\begin{pmatrix}
   x_{1}^{q^{n-1}}   & x_{2}^{q^{n-1}}   & \cdots& x_{n}^{q^{n-1}} & 0& 0&\cdots&0 \\
   x_{1}^{q^{n-2}}   & x_{2}^{q^{n-2}}   & \cdots& x_{n}^{q^{n-2}} & 0& 0&\cdots&0 \\
      \vdots & \vdots  & \cdots& \vdots & \vdots& \vdots&\cdots&\vdots \\
  x_{1}   & x_{2}  & \cdots& x_{n} & y_{1}   & y_{2}  & \cdots& y_{n} \\
  0& 0&\cdots&0 &y_{1}^{q}   & y_{2}^{q}   & \cdots& y_{n}^{q} \\
    0& 0&\cdots&0 &y_{1}^{q^{2}}   & y_{2}^{q^{2}}   & \cdots& y_{n}^{q^{2}} \\
          \vdots & \vdots  & \cdots& \vdots & \vdots& \vdots&\cdots&\vdots \\
  0& 0&\cdots&0 &y_{1}^{q^{n}}   & y_{2}^{q^{n}}   & \cdots& y_{n}^{q^{n}} \\
\end{pmatrix}$$
which is equal to  $d_{n,n}\cdot d_{n,n}^{*q}\neq 0$.
\end{proof}

\begin{prop}\label{invfield}
The invariant field 
\begin{eqnarray*}
\F(V\oplus V^{*})^{G}&=&\F(c_{n,0},u_{1-n},u_{2-n},\dots,u_{0},u_{1},\dots,u_{n-1})\\
&=&\F(c_{n,0}^{*},u_{1-n},u_{2-n},\dots,u_{0},u_{1},\dots,u_{n-1})\\
&=&\F(c_{n,0},c_{n,1},\dots,c_{n,n-1},u_{0},u_{1},\dots,u_{n-1})
\end{eqnarray*}
 is purely transcendental  over $\F$.
\end{prop}

\begin{proof}
By Bonnaf\'e-Kemper \cite[Proposition 1.1]{BK2011}, 
$\F(V\oplus V^{*})^{G}$ is generated by $$\{c_{n,0},c_{n,1},\dots,c_{n,n-1},c_{n,0}^{*},c_{n,1}^{*},\dots,c_{n,n-1}^{*},u_{0}\}$$ over $\F$.
Define $L:=\F(c_{n,0},u_{1-n},\dots,u_{-1},u_{0},u_{1},\dots,u_{n-1})$. Obviously, $L\subseteq \F(V\oplus V^{*})^{G}$ and $L$ is generated by $2n$ polynomials. Thus to prove that $\F(V\oplus V^{*})^{G}=L$ is purely transcendental,
we need only to show that $c_{n,1},\dots,c_{n,n-1},c_{n,0}^{*},c_{n,1}^{*},\dots,c_{n,n-1}^{*}\in L.$
We write the  relations (\ref{T1}), (\ref{T2}), $\dots$, (\ref{Tn-1}) as the following matrix form:
\begin{equation*}
\nabla\cdot\begin{pmatrix}
        c_{n,1}  \\
        c_{n,2}\\
        \vdots\\
        c_{n,n-1}
\end{pmatrix}=\begin{pmatrix}
      (-1)^{n+1}u_{n-1}^{q} -c_{n,0}u_{-1}  \\
       (-1)^{n+1}u_{n-2}^{q^{2}} -c_{n,0}u_{-2}  \\
       \vdots\\
      (-1)^{n+1}u_{1}^{q^{n-1}}-c_{n,0}u_{1-n}
\end{pmatrix},
\end{equation*}
where $\nabla$ is an $(n-1)\times(n-1)$-matrix whose entries lie in $\F[u_{2-n},\dots,u_{0},\dots,u_{n-2}]$.
More explicitly, $\nabla$ is the matrix given by $\nabla_{j,i} = (-1)^{i}u_{i-j}^{q^{\min(i,j)}}$ for $1 \leq i \leq n-1$ and $1 \leq j \leq n-1$.
By Lemma~\ref{algindep} we have $\det(\nabla) \neq 0$. Thus each of 
$c_{n,1},c_{n,2},\dots,c_{n,n-1}$ can be expressed as a rational expression in  
$c_{n,0},u_{1-n},\dots,u_{0},\dots,u_{n-1}$.  Therefore $c_{n,1},c_{n,2},\dots,c_{n,n-1}\in L$.
A similar argument, using the relations (\ref{T1s}), (\ref{T2s}), $\dots$, (\ref{Tn-1s}),
shows that $c_{n,1}^{*},c_{n,2}^{*},\dots,c_{n,n-1}^{*}\in L(c_{n,0}^{*})$.   Thus it remains only to show that $c_{n,0}^{*}\in L$.
This follows from the relation (\ref{T00}). This completes the proof of the first equation.

Since $\F(V\oplus V^{*})^{G}$ is $*$-stable, the second equation of the proposition holds.

To show the third equation, we let $K:=\F(c_{n,0},c_{n,1},\dots,c_{n,n-1},u_{0},u_{1},\dots,u_{n-1})$. 
Clearly $K \subseteq \F(V\oplus V^{*})^{G}$. 
We have seen that $\F(V\oplus V^{*})^{G}=L=\F(c_{n,0},u_{1-n},u_{2-n},\dots,u_{0},u_{1},\dots,u_{n-1})$.
Thus it is sufficient to show that 
$u_{-1},u_{-2},\dots,u_{1-n}\in K.$ By the relation (\ref{T1}), we have
\begin{equation*}
u_{-1}\in K
\end{equation*}
which, together with the relation (\ref{T2}), implies that 
\begin{equation*}
u_{-2}\in K.
\end{equation*}
Proceeding in this way, using the relations  (\ref{T1}), (\ref{T2}), $\dots$, (\ref{Tn-1}) (in this order) respectively, 
 we see that $u_{-1},u_{-2},\dots,u_{1-n}\in K.$
\end{proof}

\begin{prop}
$B_{\infty}\subseteq A$.
\end{prop}

\begin{proof}
It suffices to show that  $B_{k}\subseteq A$ for all $k\in\N$. By the relation (\ref{T0s}), we see that 
\begin{equation}
\label{u-n}
u_{-n}\in A.
\end{equation}
Applying the map $F^{*}$ to (\ref{T0s}), yields
\begin{equation}
\tag{$T_{-1}^{*}$}\label{T-1s}
c_{n,0}^{*q}u_{-1}-c_{n,1}^{*q}u_{-2}+c_{n,2}^{*q}u_{-3}-\cdots+(-1)^{n-1}c_{n,n-1}^{*q}u_{-n}+(-1)^{n}u_{-n-1}=0
\end{equation}
which, together with (\ref{u-n}), shows that  $u_{-n-1}\in A$. 
Continuing to apply $F^*$ in this manner, we obtain
\begin{equation}
\tag{$T_{-j}^{*}$}\label{T-js}
c_{n,0}^{*{q^j}}u_{-j}-c_{n,1}^{*{q^j}}u_{-j-1}+c_{n,2}^{*{q^j}}u_{-j-2}-\cdots+(-1)^{n-1}c_{n,n-1}^{*{q^j}}u_{-j-n+1}
       +(-1)^{n}u_{-j-n}=0
\end{equation}
for all $j \geq 1$.
Thus 
$$u_{-n},u_{-n-1},\dots,u_{-n-k+1}\in A.$$
Therefore, $B_{k}\subseteq A$.
\end{proof}

The corollary below follows immediately.
\begin{coro}\label{coro}
For any $k\in\N$, we have
 $\Qt(B_{k})=\Qt(B_{\infty})=\Qt(A)=\F(V\oplus V^{*})^{G}$.
 \end{coro}

\subsection{Localized polynomial ring} This section is devoted to showing that localizing 
$\F[V\oplus V^{*}]^{G}$ with respect to the multiplicative set $\{c_{n,0}^m \mid m \in \N\}$
 yields is a localized polynomial ring.  We begin by proving a key lemma.

\begin{lem}\label{key lemma}
  Let $\alpha =  f_{1}^{a_{1}}f_2^{a_2}\cdots f_{n}^{a_{n}}\cdot f_{1}^{*b_{1}}f_2^{*b_2}\cdots f_{n}^{*b_{n}}$.\\
  There exists $r_0 \in \N$ such that $c_{n,0}^r \,\alpha \in 
  \F[f_1,f_2,\dots,f_n,u_{1-n},u_{2-n},\dots,u_{n-1}]$ for all $r \geq r_0$.\\
 Furthermore $c_{n,0}^r R_U^G(\alpha) \in \F[c_{n,0},c_{n,1}\dots,c_{n,n-1},u_{1-n},u_{2-n},\dots,u_{n-1}]$
 for all $r \geq r_0$.
\end{lem}

\begin{proof}
  The following relation is from Bonnaf\'e-Kemper \cite[page 105]{BK2011}.
  \begin{equation*} \tag{$R_k$} \label{Rk} 
  f_k \cdot f^*_{n+1-k} - \sum_{i=0}^{k-1} \sum_{j=0}^{n-k} (-1)^{i+j+n+1}c_{k-1,j}\, c^*_{n-k,j}\cdot u_{i-j}^{q^{\min(i,j)}} = 0.
  \end{equation*}
  Here $c_{k-1,j} \in \F[f_1,f_2,\dots,f_{k-1}] \subset \F[x_1,x_2,\dots,x_{k-1}]$ and
  $c^*_{n-k,j} \in \F[f^*_1,f^*_2,\dots,f^*_{n-k}] \subset \F[y_n,y_{n-1},\dots,y_{k+1}]$.
  
  The relation $(R_{k})$ implies that
\begin{equation}
\label{Rii}
f_{n+1-i}\cdot f^{*}_{i}\in \F[f_{1},\dots,f_{n-i},f_{1}^{*},\dots,f_{i-1}^{*},u_{n-1},\dots,u_{0},\dots,u_{1-n}]
\end{equation}
for all $i=1,2,\dots,n$. Since $c_{n,0}=(f_{1}f_{2}\cdots f_{n})^{q-1}$,  there exists some 
$r_0\in\N$ such that
\begin{equation*}
c_{n,0}^{r_0}\cdot\alpha\in \F[f_{1},\dots,f_{n},u_{n-1},\dots,u_{0},\dots,u_{1-n}].
\end{equation*}

The second assertion follows from
\begin{align*}
    c_{n,0}^r R_U^G(\alpha) = R_U^G(c_{n,0}^r \alpha) 
 &\in R_U^G(\F[f_1,f_2,\dots,f_n,u_{1-n},u_{2-n},\dots,u_{n-1}])\\ 
 &= \F[c_{n,0},c_{n,1},\dots,c_{n,n-1},u_{1-n},u_{2-n},\dots,u_{n-1}].
 \end{align*}
\end{proof}

\begin{prop}\label{local}
 $\F[V\oplus V^{*}]^{G}[c_{n,0}^{-1}]=A[c_{n,0}^{-1}]=\F[c_{n,0},c_{n,1},\dots,c_{n,n-1},u_{0},u_{1},\dots,u_{n-1}][c_{n,0}^{-1}]$.
 \end{prop}

\begin{proof}
Let $E:=\F[c_{n,0},c_{n,1},\dots,c_{n,n-1},u_{0},u_{1},\dots,u_{n-1}]$ be the polynomial ring. For the first assertion, 
we note that the Reynolds operator $R_{U}^{G}:\F[V\oplus V^{*}]^{U}\ra \F[V\oplus V^{*}]^{G}$ is a surjective $A$-module homomorphism.
By (\ref{Amodule}), $\F[V\oplus V^*]^G = \sum_{\omega \in \Omega} A\cdot R_U^G(\omega)$.
Lemma~\ref{key lemma} implies that there exists $r \in \N$ such that
$c_{n,0}^r R_U^G(\omega) \in  \F[c_{n,0},c_{n,1}\dots,c_{n,n-1},u_{1-n},u_{2-n},\dots,u_{n-1}] \subset A$.
 Therefore $R_U^G(\omega) \in A[c_{n,0}^{-1}]$ for all $\omega \in \Omega$.
 Thus $\F[V\oplus V^*]^G \subset A[c_{n,0}^{-1}]$.

To show that $A[c_{n,0}^{-1}]=E[c_{n,0}^{-1}]$, it is sufficient to show that $u_{-1},u_{-2},\dots,u_{1-n}, c_{n,0}^{*}, c_{n,1}^{*},\dots,c_{n,n-1}^{*}\in E[c_{n,0}^{-1}]$. By the relation (\ref{T1}) we see that $u_{-1}\in E[c_{n,0}^{-1}]$, which together with 
the relation (\ref{T2}), implies that $u_{-2}\in E[c_{n,0}^{-1}]$. 
Proceeding in this way by using the relations (\ref{T1}), (\ref{T2}), \dots, (\ref{Tn-1}), we obtain 
\begin{equation}
\label{ue}
u_{-1},u_{-2},\dots,u_{1-n}\in E[c_{n,0}^{-1}].
\end{equation}
Since $c_{n,0}^{*}, c_{n,1}^{*},\dots,c_{n,n-1}^{*} \in \F[V^*]^U = \F[f^*_1,f^*_2,\dots,f^*_n]$
 it follows from Lemma~\ref{key lemma} that there exists some $k\in \N$ such that 
\begin{equation*}
c_{n,0}^{k}\cdot c_{n,i}^{*}\in \F[f_{1},\dots,f_{n},u_{n-1},\dots,u_{0},\dots,u_{1-n}].
\end{equation*}
for $1\leqslant i\leqslant n$. Thus
\begin{equation*}
c_{n,0}^{k}\cdot c_{n,i}^{*}=R_{U}^{G}(c_{n,0}^{k}\cdot c_{n,i}^{*})\in \F[c_{n,0},c_{n,1},\dots,c_{n,n-1},u_{n-1},\dots,u_{0},\dots,u_{1-n}].
\end{equation*}
This, together with (\ref{ue}) implies that 
$c_{n,0}^{*}, c_{n,1}^{*},\dots,c_{n,n-1}^{*}\in E[c_{n,0}^{-1}]$.
 Therefore $A[c_0^{-1}] \subseteq E[c_0^{-1}]$.
\end{proof}

\subsection{Unique Factorization Domains} 

The bulk of this section will be devoted to proving that $B_k$  is a complete intersection and a unique factorization domain.  

We let $\Frob$ denote the Frobenius homorphism given by $\Frob(z) = z^p$.
We will use the following lemma which is easy to prove.
  \begin{lem}
     Let $\KK$ be a field of characteristic $p$ and let $\lambda \in \KK$ and let $t$ be a positive integer.  The polynomial $x^{p^t}-\lambda$ is reducible in $\KK[x]$ if and only if there exists $\nu \in \KK$ such that
     $\nu^p=\lambda$, i.e., if and only if $\lambda \in \Frob(\KK)$. 
  \end{lem}

The main result of this section is the following.
\begin{prop}\label{Bk}
 $B_{k}$ is a complete intersection and a unique factorization domain for all $k \geq 0$,
\end{prop}

\begin{proof} 
First we will prove that $B_k$ is a complete intersection.

We introduce formal variables $C_{0},C_{1},\dots,C_{n-1},U_{n-1},U_{n-2},\dots,U_{0},\dots,U_{-n-k+1}$
and define a map
$$\rhoup: S_k=\F[C_{0},C_{1},\dots,C_{n-1}, U_{n-1},\dots,U_{0},\dots,U_{-n-k+1}]\ra B_{k}$$
by $\rhoup(C_{i}) = c_{n,i}$  and $\rhoup(U_{j}) = u_{j}$.

The relation (\ref{T1}) corresponds to an element (\ref{Tn1})
in the kernel of $\rhoup$.  
Applying $F^*$ repeatedly to (\ref{T1}) yields further relations in $B_k$ and
so also the following corresponding elements of the kernel of $\rhoup$.

\begin{align*}
\tag{$T^n_{1}$} \label{Tn1}
C_0 U_{-1} - C_{1} U_{0}^{q} + C_{2} U_{1}^{q}& - \dots+(-1)^{n-1} C_{n-1} U_{n-2}^{q}+(-1)^{n}U_{n-1}^q \\
\tag{$T^n_{2}$} \label{Tn2}
C_0 U_{-2} - C_{1} U_{-1}^{q} + C_{2} U_{0}^{q^2}& - \dots+(-1)^{n-1} C_{n-1} U_{n-3}^{q^2}+(-1)^{n}U_{n-2}^{q^2}\\
&\vdots\\
 \tag{$T^n_{n}$} \label{Tnn}
C_{0} U_{-n} - C_{1}U_{1-n}^{q}+C_{2}U_{2-n}^{q^{2}}&-\dots+(-1)^{n-1}C_{n-1}U_{-1}^{q^{n-1}}+(-1)^{n}U_{0}^{q^{n}}\\
\tag{$T^n_{n+1}$} \label{Tnn+1}
C_{0}U_{-(n+1)}-C_{1}U_{-n}^{q}+C_{2}U_{1-n}^{q^{2}}&-\dots+(-1)^{n-1}C_{n-1}U_{-2}^{q^{n-1}}+(-1)^{n}U_{-1}^{q^{n}}\\
&\vdots\\
\tag{$T^n_{n+k-1}$}\label{Tnn+k-1}
C_{0}U_{-(n+k-1)}-C_{1}U_{-(n+k)+2}^{q}+C_{2}U_{-(n+k)+3}^{q^{2}}&-\dots+(-1)^{n-1}C_{n-1}U_{-k}^{q^{n-1}}+(-1)^{n}U_{-(k-1)}^{q^{n}}.
\end{align*}

In general we have 
\begin{equation}
\tag{$T^n_{j}$}\label{Tnj}
  \sum_{i=0}^{n-1} (-1)^i C_i U_{-j+i}^{q^{\min(i,j)}} + (-1)^n U_{n-j}^{q^{\min(n,j)}} 
\end{equation}
for all $j \geq 1$.

Thus we have a surjective $\F$-algebra homomorphism
$$\rhoup:\Sb{n}{k}=S/(T^n_{1},\dots, T^n_{n+k-1})\ra B_{k}.$$

\textsc{Claim 1}: $\Sb{n}{k}$ is an integral domain.

In the following we will make repeated use of the fact that permutating the order of a regular sequence consisting of homogeneous elements preserves the regularity.
Note that 
\begin{eqnarray*}
T^n_{1} & \equiv & (-1)^{n} U_{n-1}^{q} \mod (C_{0},C_{1},\dots,C_{n-1})\\
T^n_{2} & \equiv &(-1)^{n} U_{n-2}^{q^{2}} \mod (C_{0},C_{1},\dots,C_{n-1})\\
&\vdots&\\
T^n_{n} & \equiv & (-1)^{n} U_{0}^{q^{n}} \mod (C_{0},C_{1},\dots,C_{n-1})\\
T^n_{n+1} & \equiv &(-1)^{n} U_{-1}^{q^{n}}  \mod (C_{0},C_{1},\dots,C_{n-1})\\
&\vdots&\\
T^n_{n+k-1} & \equiv &(-1)^{n} U_{-k+1}^{q^{n}}  \mod (C_{0},C_{1},\dots,C_{n-1}).
\end{eqnarray*}
Since $S_k$ is a polynomial algebra, $$C_{0},C_{1},\dots,C_{n-1},U_{n-1},\dots,U_{1},U_{0},U_{-1},\dots,U_{-n-k+1}$$
is a regular sequence in $S_k$, as is
$$C_{0},C_{1},\dots,C_{n-1},U_{-n-k+1},\dots,U_{-k},U_{-k+1}^{q^{n}},\dots,U_{0}^{q^{n}},U_{1}^{q^{n-1}},\dots, U_{n-1}^{q}.$$
Thus $C_{0},\dots,C_{n-1},U_{-n-k+1},\dots,U_{-k},T^n_{n+k-1},\dots,T^n_{n},T^n_{n-1},\dots,T^n_{1}$ 
is a regular sequence in $S_k$.
This means that $C_{0}$ is not a zerodivisor in $\Sb{n}{k}$ and $\Sb{n}{k}$ can be embedded into $\Sb{n}{k}[C_{0}^{-1}]$. Moreover, in 
$\Sb{n}{k}[C_{0}^{-1}]$, the relations (\ref{Tn1}), (\ref{Tn2}), \dots, (\ref{Tnn+k-1}) (in this order) imply that 
$$U_{-1},U_{-2},\dots, U_{-n-k+1}$$ can be expressed in terms of 
$C_{0},C_{1},\dots,C_{n-1},U_{0},U_{1},\dots,U_{n-1}$ and $C_{0}^{-1}$. This implies 
$$\Sb{n}{k}[C_{0}^{-1}]\cong\F[C_{0},C_{1},\dots,C_{n-1},U_{0},U_{1},\dots,U_{n-1}][C_{0}^{-1}].$$
Therefore, $\Sb{n}{k}$, as a subring of $\Sb{n}{k}[C_{0}^{-1}]$,  is an integral domain.

\textsc{Claim 2}: The map $\rhoup$ is an isomorphism and so $B_{k}$ is a complete intersection. 

By claim 1, $\Sb{n}{k}$ is integral domain, so  the height of $\ker(\rhoup)$ is equal to $\dim(\Sb{n}{k})-\dim(B_{k})=0$. Since $B_{k}$ is an integral domain, $\ker(\rhoup)$ is a prime ideal in $\Sb{n}{k}$. Note that
$\{0\}\subseteq \ker(\rhoup)$ and $\{0\}$ is a prime ideal in $\Sb{n}{k}$, so $\ker(\rhoup)=0$, i.e., $\rhoup$ is injective.  
Therefore,  $\rhoup$ is an isomorphism and $B_{k}$ is a complete intersection.

Furthermore the isomorphism $\Sb{n}{k}[C_{0}^{-1}]\cong\F[C_{0},C_{1},\dots,C_{n-1},U_{0},U_{1},\dots,U_{n-1}][C_{0}^{-1}]$ shows that $\Sb{n}{k}[C_{0}^{-1}]$ is a unique factorization domain.
In order to to show that $\Sb{n}{k}$ is a unique factorization domain, by Nagata's Lemma (see Eisenbud \cite[Lemma 19.20]{Eis1995}),  it is sufficient to show that $C_0$ is a prime element in $\Sb{n}{k}$.
 Thus it only remains to show that $\Sb{n}{k}/(C_0)$ is a domain 
  for all $n \geq 2$ and for all $k \geq 0$.

  Consider the image $\overline{T^n_j}$ in $\F[C_0,C_1,\dots,C_{n-1},U_{n-1},\dots,U_{-n-k+1}]/(C_0)$
  of $T^n_j$.  We express $\overline{T^n_j}$ in terms of the variables
  $\tilde{C}_i := C_{i+1}$ for $i=0,1,\dots,n-2$ and the variables
  $\tilde{U}_j$ defined by 
  $\tilde{U}_j := \begin{cases} U_{j+1},& \text{if } j \geq 0;\\
                                               U_{j+1}^q,& \text{if } j \leq -1 \end{cases}$
                                               for $-n-k+1 \leq j \leq n-1$.
  In terms of these variables $-\overline{T^n_j}$ is expressed as
  \begin{equation}
  \tag{$\tilde{T}_j$} \label{tildeTj} \sum_{i=0}^{n-1} (-1)^i \tilde{C}_i \tilde{U}_{-j+i}^{q^{\min(i,j)}}.
  \end{equation}                                         
    Note that the relations (\ref{tildeTj}) for $j=1,2,\dots,-n-k+1$ are precisely the relations $T^{n-1}_j$ used to define $\Sb{n-1}{k+1}$, 
    i.e, $$\Sb{n-1}{k+1} \cong 
 \frac
  {\F[\tilde{C}_0,\tilde{C}_1,\dots,\tilde{C}_{n-2},\tilde{U}_{n-2},\dots,\tilde{U}_{-n-k+1}]}
  {(\tilde{T}_1,\tilde{T}_2,\dots,\tilde{T}_{-k-n+1})}$$
  where the $\tilde{C_i}$ and the $\tilde{U_j}$ are indeterminants.
  
    Recalling that in $\Sb{n}{k}$ the $\tilde{U_j}$ are not all indeterminants but rather
    $\tilde{U_j} = U_{j+1}^q$ if $j \leq -1$  we have that                                       
$$\Sb{n}{k}/(C_0) \cong \frac{\Sb{n-1}{k+1}[Z_{-1},Z_{-2},\dots,Z_{-n-k+1}]}
          {(Z_{-1}^q - \tilde{U}_{-1}, Z_{-2}^q - \tilde{U}_{-2}, \dots, Z_{-n-k+1}^q - \tilde{U}_{-n-k+1})}.$$

  Since $\Sb{n-1}{k+1}[C_0^{-1}] \cong \F[C_0,C_1,\dots,C_{n-2},U_{n-2},\dots,U_0][C_0^{-1}]$,
  we see that $\Sb{n-1}{k+1}$ is a domain.

Let $Q_0$ denote the quotient field $Q_0 := \Qt(\Sb{n-1}{k+1})$.  The inclusion of 
$\Sb{n-1}{k+1} \hookrightarrow Q_0$ induces a natural inclusion
$$ \frac{\Sb{n-1}{k+1}[Z_{-1},Z_{-2},\dots,Z_{-n-k+1}]}
          {(Z_{-1}^q - \tilde{U}_{-1}, Z_{-2}^q - \tilde{U}_{-2}, \dots, Z_{-n-k+1}^q - \tilde{U}_{-n-k+1})}
\hookrightarrow \frac{Q_0[Z_{-1},Z_{-2},\dots,Z_{-n-k+1}]}
          {(Z_{-1}^q - \tilde{U}_{-1}, Z_{-2}^q - \tilde{U}_{-2}, \dots, Z_{-n-k+1}^q - \tilde{U}_{-n-k+1})}.$$
          We will show that this latter ring is in fact a field.
          
          Define $Q_1 := Q_0[Z_{-1}]/(Z_{-1}^q - \tilde{U}_{-1})$.  Since 
          $\tilde{U}_{-1} \notin \Frob(Q_0)$, the monic polynomial $Z_{-1}^q - \tilde{U}_{-1}$ is irreducible
          in $Q_0[Z_{-1}]$, and thus $Q_1$ is a field.
          
          Similarly we define $Q_2 := Q_1[Z_{-2}]/(Z_{-2}^q - \tilde{U}_{-2})$.  Again since 
          $\tilde{U}_{-2} \notin \Frob(Q_1)$, we see that $Q_2$ is a field.
         
           Continuing in this manner we define fields $Q_{i} := Q_{i-1}[Z_{-i}]/(Z_{-i}^q - \tilde{U}_{-i})$
           for $i=1,2,\dots,n+k-1$.
          Furthermore, for each $i \geq 1$, we have
           $$Q_i \cong \frac{Q_0[Z_{-1},Z_{-2},\dots,Z_{-i}]}
          {(Z_{-1}^q - \tilde{U}_{-1}, Z_{-2}^q - \tilde{U}_{-2}, \dots, Z_{-i}^q - \tilde{U}_{-i})}.$$
        Since the          
           ring $$\Sb{n}{k}/(C_0) \cong \frac{\Sb{n-1}{k}[Z_{-1},Z_{-2},\dots,Z_{-n-k+1}]}
          {(Z_{-1}^q - \tilde{U}_{-1}, Z_{-2}^q - \tilde{U}_{-2}, \dots, Z_{-n-k+1}^q - \tilde{U}_{-n-k+1})}$$ 
          injects into the field $Q_{n+k-1}$, it is a domain as required.

  Since $C_0$ is a prime element of $\Sb{n}{k}$ and $\Sb{n}{k}[C_0^{-1}]$ is a unique factorization domain 
  it follows that $\Sb{n}{k}$ is also a unique factorization domain.  Finally the isomorphism 
  $\rhoup: \Sb{n}{k} \to B_k$ completes the proof of Proposition~\ref{Bk}.
  \end{proof}

Next we consider $B_\infty$.

\begin{prop}\label{Binf}
$B_{\infty}$ is a unique factorization domain. 
\end{prop}

\begin{proof}
Fix a degree $d\in\N^{+}$, we note that $(B_{\infty})_{d}=(B_{k})_{d}$ for all $k\gg 0$.
For any element $a\in B_{\infty}$, we can choose $d$ such that $a\in X:=\oplus_{i=0}^{d}(B_{\infty})_{i}$. Then there exist $k_{0}\in\N$ such that $X=\oplus_{i=0}^{d}(B_{\infty})_{i}=\oplus_{i=0}^{d} (B_{k_{0}})_{i}$. 
Then any factorization of $a$ into primes 
in $B_{\infty}$ will take place in $X$ and thus also in $B_{k_{0}}$. By Proposition \ref{Bk}, any factorization of $a$ into primes 
in $B_{k_{0}}$ (so in $B_{\infty}$) is unique, i.e., $B_{\infty}$ is a unique factorization domain. 
\end{proof}

\subsection{Almost-integral elements and pseudo-almost integral elements} 

Let $D$ be an integral domain. An element $\betaup\in\Qt(D)$ is said to be \textit{almost-integral} over $D$ if there exists a nonzero element $c\in D$ such that $c\cdot \betaup^{m}\in D$ for any $m\in \N^{+}.$ 
Any integral element over $D$ is almost-integral but the converse does not necessarily holds in general. 
If $D$ is noetherian, an almost-integral element over $D$ is also an integral element. We say that $D$ is \textit{completely integrally closed} if any almost-integral element belongs to $D$.

In 2007, Anderson-Zafrullah \cite{AZ2007} introduced the notion of \textit{pseudo-almost integral elements}.
For an integral domain $D$, an element $\betaup\in\Qt(D)$ is \textit{pseudo-almost integral} over $D$ if there exists an 
infinite set of positive integers $\{s_{k} \mid k \in \N\}$ and
 a nonzero element $c\in D$ such that $c\cdot \betaup^{s_{k}}\in D$ for all  $s_{k}.$
Clearly, for any integral domain, $\{\textrm{integral elements}\}\subseteq \{\textrm{almost-integral elements}\}\subseteq \{\textrm{pseudo-almost integral elements}\}$.

Recall that $h=R_U^G(\omega)$ where $\omega \in \Omega$. 
 By Corollary \ref{coro}, we see
that $h\in \F(V\oplus V^{*})^{G}=\Qt(B_{\infty})$. Moreover, 

\begin{prop}\label{pai}
$h$ is pseudo-almost integral over $B_{\infty}$.
\end{prop}

\begin{proof}
We write $\omega=\omega_{x}\cdot\omega_{y}$
where $\omega_{x}=f_{1}^{a_1}f_{2}^{a_2}\cdots f_{n}^{a_n}$ and 
$\omega_{y}=f_{1}^{*b_1}f_{2}^{*b_2}\cdots f_{n}^{*b_n}$.

Let $s_{k}:=q^{k}$ with $k\in\N$, then $\{s_{k}\mid k\in\N\}$ is  an infinite set of positive integers. 
We will show that there exists $r_0 \in \N$ such that for all $k \geq 0$ we have $c_{n,0}^{r_0} h^{q^k} \in B_{\infty}$.
By the first assertion of Lemma~\ref{key lemma}, there exists $r_0 \in N$ such that
$c_{n,0}^{r_0}\, \omega =P(f_1,f_2,\dots,f_n,u_{n-1},\dots,u_{1-n}) \in \F[f_1,f_2,\dots,f_n,u_{n-1},\dots,u_{1-n}]$.
Applying $(F^*)^k$ yields 
\begin{align*}
c_{n,0}^{r_0}\,\omega_x \cdot \omega_y^{q^k} &= P(f_1,f_2,\dots,f_n,(F^*)^k(u_{n-1}),\dots,(F^*)^k(u_{1-n}))\\
  &= P(f_1,f_2,\dots,f_n,u_{n-1-k},\dots,u_{1-n-k})
\end{align*}

Thus 
\begin{align*}
c_{n,0}^{r_0}\, \omega^{q^k} &= c_{n,0}^{r_0} \omega_x^{q^k-1} \left(\omega_x \omega_y^{q^k} \right)\\
         &= \omega_x^{q^k-1} P(f_1,f_2,\dots,f_n,u_{n-1-k},\dots,u_{1-n-k})  \text{ and}\\
 c_{n,0}^{r_0}\, h^{q^k} &= c_{n,0}^{r_0} R_U^G(\omega)^{q^k} = c_{n,0}^{r_0}R_U^G(\omega^{q^k}) \\
      &=  c_{n,0}^{r_0}\, R_U^G(\omega_x^{q^k-1}\cdot P(f_1,f_2,\dots,f_n,u_{n-1-k},\dots,u_{1-n-k})) \in B_k
\end{align*}
Hence $c_{n,0}^r h^{s_k} \in B_{k} \subseteq B_{\infty}$ for all $k \geq 0$ as desired. 
\end{proof}

\subsection{Proof of Theorem \ref{mt}} 

Now we complete the proof that $A=\F[V\oplus V^*]^G$.

\begin{proof}[Proof of the first assertion of Theorem \ref{mt}]
By Section 5, it suffices to show that
$h=R_{U}^{G}(\omega)\in A$.   
We will show that $h \in B_\infty \subset A$.
 It is not hard to show directly that as unique factorization domain $B_\infty$ contains all the pseudo-almost integral
elements in its field of fractions.  Alternatively we reason as follows.
%
Recall that a domain $D$ is \textit{root closed} if for $\betaup\in \Qt(D)$ with $\betaup^{n}\in D$ for some $n\in\N^{+}$, then $\betaup\in D$.
Since $B_\infty$ is a unqiue factorization domain, it is integrally closed, so is root closed. Since $h$ is psuedo-almost integral over 
$B_\infty$, 
Anderson-Zafrullah \cite[Proposition 2]{AZ2007} implies that $h$ is almost-integral over $B_{\infty}$.
Since any unique factorization domain is completely integrally closed (see for example, Huneke-Swanson \cite[Exercise 2.26 (vii)]{HS2006}), 
$B_{\infty}$ is completely integrally closed and thus $h\in B_{\infty}\subseteq A$.
\end{proof}

Next we show that we have a minimal set of generators for $\F[V \oplus V^*]^G$.
 The algebra $\F[V \oplus V^*]$ is naturally graded by $\N \oplus \N$ where $\deg(x_i)=(1,0)$ and $\deg(y_i)=(0,1)$ for $1\leq i \leq n$.   We refer to this grading as the
  ``bigrading'' on $\F[V \oplus V^*]$.  Since the action of $G$ preserves this bigrading, the bigrading restricts to a bigrading of $\F[V \oplus V^*]^{G}$.

  Let 
  $$\Lambda := \{c_{n,0},c_{n,1},\dots,c_{n,n-1},c_{n,0}^{*},c_{n,1}^{*},\dots,c_{n,n-1}^{*},u_{1-n},u_{2-n},\dots,u_{0},\dots,u_{n-2},u_{n-1}\}$$ be the given set of generators of $\F[V \oplus V^*]^G$.
  Note that $\dim\Big(\F[V \oplus V^*]^{G}_{(d_1,d_2)}\Big)=1$ for $(d_1,d_2) = \deg(z)$ for all $z \in \Lambda$,
 i.e., for the bidgrees $\{(q^n-q^i,0), (0,q^n-q^i), (q^i,1), (1,q^i) \mid 0 \leq i \leq n-1\}$.
 Since each these graded components is 1-dimensional, it follows that the $4n-1$ elements of $\Lambda$ {\em minimally} generate
 $\F[V \oplus V^*]^{G}$.

  It remains to prove the last two assertions of Theorem \ref{mt}.
  
  \begin{proof}[Proof of the second assertion of Theorem \ref{mt}]
  If we represent the action on $V$ of an element $\sig \in G$ by an $n\times n$ matrix $E$ then the action of
 $\sig$ on $V \oplus V^*$ is given by the matrix $B := E \oplus (E^{-1})^t$.    Clearly $\det(B) =1$ and
$\rank(\text{I}_{2n} - B) = 2\cdot \rank(\text{I}_n - E)$.  In particular, $\rank(\text{I}_{2n} - B) \neq 1$, i.e., $B$ is never a pseudo-reflection.  Bonnaf\'e-Kemper \cite[Theorem 2.4]{BK2011} has proved $\F[V \oplus V^*]^{U}$ is a complete intersection, so is Cohen-Macualay.  Thus
$\F[V \oplus V^*]^{G}$ is Cohen-Macaulay by Campbell-Hughes-Pollack \cite[Theorem 1]{CHP1991}.  Since $\F[V\oplus V^*]^{G}$ is Cohen-Macaulay, $G \leq \SL(V\oplus V^*)$ and the action of $G$ on $V\oplus V^*$ is pseudo-reflection free, it follows by Braun \cite[Theorem]{Bra2011} that 
$\F[V \oplus V^*]^{G}$ is Gorenstein.
\end{proof}

  Finally we will show that $\F[V \oplus V^*]^{G}$ is not a complete intersection.

  Consider the polynomial ring
  $S$ on $4n-1$ formal variables
  $$S = \F[C_0,C_1,\dots,C_{n-1},C^*_0,C^*_1,\dots,C^*_{n-1},U_{1-n},\dots,U_0,\dots,U_{n-1}].$$
  We have a surjection of algebras $\piup:S \ra \F[V\oplus V^*]^{G}$ defined by $\piup(C_i) = c_{n,i}$, $\piup(C^*_i)=c^*_{n,i}$
  for $0\leq i \leq n-1$ and $\piup(U_i)=u_i$ for $1-n\leq i \leq n-1$.

We may pull back, via $\piup$, the bigrading on $\F[V \oplus V^*]^{G}$ to get an induced bigrading on $S$ which makes $\piup$ a morphism of $\N\oplus \N$-graded algebras.
  Let $K$ denote the kernel of $\piup$.  Then $K$ is a homogeneous ideal with respect to the bigrading.    Furthermore the involution $*$ and the homomorphisms $F$ and $F^*$ also pull back to endomorphisms of $S$ which we will also denote by $*$, $F$ and $F^*$.

For $j=1,2,\dots,n-1$ we write $T_{j}$ for  the elements of $S$ corresponding to the left hand side of the 
relation~(\ref{Tj}), i.e.,
  \begin{equation*} \label{Tsubj}
       T_j := \sum_{i=0}^n (-1)^i C_i U_{i-j}^{q^{\min\{i,j\}}} 
  \end{equation*} 
   for $j =1,2,\dots,n-1$.
  Here we use $C_{n}=C_{n}^{*}=1$ since $c_{n,n}=c^*_{n,n}=1$.
  
We also define $T_{j}^{*}:=*(T_{j})\in S$ for $j = 1,2,\dots, n-1$. 
Similarly we write $T_{00}$ for the element of $S$ corresponding to the left hand side of the relation~(\ref{T00}).
By Lemma \ref{ini} and the discussion following it, we see that

\begin{prop}
$T_{00}$, $T_1,T_2,\dots,T_{n-1}$, $T^*_1,T^*_2,\dots, T^*_{n-1}$ are elements of $K$, the kernel of $\piup$.
\end{prop}

  \begin{rem}{\rm
     We may use $T_0$ and $T^*_0$ to define in $S$ elements $U_n$ and $U_{-n}$ respectively.  Then $T_n$ and $T^*_n$ yield 
     elements of $K$.  
     Note that $T_n-T_n^* = (-1)^{n+1} \sum_{j=0}^{n-1} (C^*_j T_j - C_j T^*_j) \in (T_1,\dots,T_{n-1},T^*_1,\dots,T^*_{n-1})$.
         The equation  
     $\sum_{j=0}^{n} (-1)^{j}(C^*_j T_j - C_j T^*_j)=0$ is straight forward to verify.  
}\end{rem}

\begin{proof}[Proof of the third assertion of Theorem \ref{mt}]
  Let $S_+$ denote the ideal $S_+ := \displaystyle\bigoplus_{d_1 +d_2 \geq 1} S_{(d_1,d_2)}$ 
  so that $S = \F \oplus S_+$.
   We consider the $2n-1$ elements $T_{00}$, $T_1,\dots,T_{n-1}$, $T^*_1,\dots, T^*_{n-1}$ of $K$.
  Note that for $j \geq 1$ we have $T_j \equiv C_0 U_{-j} \pmod{S_+^3}$ (unless $(j,q)=(1,2)$ in which case
  $T_1 \equiv C_0 U_{-1} + (-1)^n U_{n-1}^2 \pmod{S_+^3}$).
   Similarly the element $T_{00}$ of $S$ satsifies $T_{00} \equiv C_0 C^*_0\pmod{S_+^3}$
    (unless $(n,q)=(2,2)$ in which case $T_{00} \equiv C_0 C^*_0  + U_1U_{-1}  \pmod{S_+^3}$).

        The fact that none of these $2n-1$ elements lies in the ideal $S_+^3$,
  together with the fact that these elements all have different bidegrees
  implies that some non-zero scalar multiple of each of these $2n-1$ elements lies in every homogeneous minimal 
  generating set for $K$.  
  Therefore $\F[V \oplus V^*]^{G}$
  is a complete intersection if and only if $K$ is generated by the $2n-1 = \dim S - \dim \F[V\oplus V^*]^G$ elements
  $T_{00}$, $T_1,T_2,\dots,T_{n-1}$, $T^*_1,T^*_2, \dots, T^*_{n-1}$.
  Assume, by way of contradiction, that $\F[V\oplus V^*]^G$ is a complete intersection and thus these elements do generate $K$.
  
  By Stanley \cite[Corollary 3.3]{St1978}, this implies that 
 that the Hilbert series of $\F[V \oplus V^*]^{G}$
    is given by
    \begin{eqnarray*}
    \HH\Big(\F[V \oplus V^*]^{G},\lam\Big) &=&  \frac{ (1-\lam^{2q^n-2})\cdot\prod_{k=1}^{n-1} (1-\lam^{q^n+q^k})^2 }
          { (1-\lam^2)\cdot \prod_{i=0}^{n-1} (1-\lam^{q^n-q^i})^2 \cdot\prod_{i=1}^{n-1} (1-\lam^{q^i+1})^2}\\
               &= &\frac{\mu(2q^n-2) \prod_{k=1}^{n-1} \mu(q^n+q^k)^2 }{(1-\lam)^{2n} \mu(2)\cdot \prod_{i=0}^{n-1} \mu(q^n-q^i)^2 \cdot\prod_{i=1}^{n-1} \mu(q^i+1)^2}        
      \end{eqnarray*}
 where $\mu(m) = 1+\lam+\dots +\lam^{m-1}$.
 Now by Benson \cite[Theorem 2.4.3]{Ben1993} we know that
 $$
    \HH\Big(\F[V \oplus V^*]^{G},\lam\Big) = \frac{1}{(1-\lam)^{2n}}\left[ \frac{1}{|G|} + (1-\lam)Q(\lam) \right] \text{ for some polynomial }Q(\lam).
 $$
   Comparing these two expressions for $\HH\Big(\F[V \oplus V^*]^{G},\lam\Big)$ we see that
 $$
 \frac{1}{|G|} =
  \frac {(2q^n-2) \cdot\prod_{k=1}^{n-1} (q^n+q^k)^2 } {2\cdot \prod_{i=0}^{n-1} (q^n-q^i)^2 \cdot\prod_{i=1}^{n-1} (q^i+1)^2}
  $$  
 Therefore
 $$
  |G| =  \frac {\prod_{i=0}^{n-1} (q^n-q^i)^2 \cdot\prod_{i=1}^{n-1} (q^i+1)^{2}} {(q^n-1)\cdot\prod_{k=1}^{n-1} (q^n+q^k)^2 }.
 $$ 
  The highest power of $q$ dividing the numerator of the last expression is $2\sum_{i=0}^{n-1} i = n^2-n$. 
  Also the highest power of $q$ dividing the  denominator is 
  $2\sum_{k=1}^{n-1} k = n^2-n$.
   Thus $q$ does not divide 
this fraction.   But of course $q$ does divide the order of $G=GL(V)$.
 This contradiction shows that $\F[V \oplus V^*]^{G}$ is not a complete intersection.
\end{proof}

\subsection{A conjecture} 

  We have attempted to describe minimal generators for the ideal $K$ described in the previous section.  Our experiments using the computer algebra system
  MAGMA \cite{magma} lead us to make the following conjecture.

  \begin{conj}
  Consider the algebra epimorphism $\piup:S\ra \F[V \oplus V^*]^{G}$ as above. Then $K$,
  the kernel of $\piup$, is minimally generated by the $2(n-1)$ elements $T_{1},\dots,T_{n-1},T_{1}^{\ast},\dots,T_{n-1}^{\ast}$ together with $\binomial{n+1}{2}$ other elements 
  $T_{i,j}$  where $0 \leq i,j \leq n-1$ and $i+j \leq n-1$ satisfying $$T_{i,j} \equiv C_i\cdot C^*_j \mod {S_+^3}$$
  (for $q \geq 3$).
  \end{conj}
  
  We have been able to verify this conjecture computationally using MAGMA for values of 
  $(n,q) \in \{(2,q) \mid q \leq 16\} \cup \{(3,q) \mid q \leq 7\} \cup \{ (4,2), (4,3), (5,2)\}$.


\section*{\textit{Acknowledgments}}

The first author 
thanks the Department of Mathematics and Statistics of Queen's University for providing a comfortable working environment
during his visit in 2014--2016.  We thank the anomymous referee of an earlier version of this paper for pointing out an 
error in that earlier version.  We also thank Gregor Kemper for his corrections and comments.
The first author was partially supported by NSF of China (No. 11401087) and China Scholarship Council (No. 201406625007).
Both authors were partially supported by NSERC.



\begin{thebibliography}{99}

\bibitem{AZ2007}  \textsc{D. Anderson \& M. Zafrullah}, Pseudo-almost integral elements. \textit{Comm. Algebra} 35 (2007) 1127--1131.

\bibitem{Ben1993} \textsc{D. Benson}, Polynomial Invariants of Finite Groups,
\textit{London Mathematical Society Lecture Note Series}190. Cambridge University Press, Cambridge, 1993. x+118 pp. ISBN: 0-521-45886-2.

\bibitem{BK2011}  \textsc{C. Bonnaf\'e \& G. Kemper}, Some complete intersection symplectic quotients in positive characteristic: invariants of a vector and a covector. \textit{J. Algebra} 335 (2011) 96--112.
    
\bibitem{magma}
\textsc{W. Bosma, J. Cannon \& C. Playoust}, The Magma algebra system I: the user language.
\textit{J. Symbolic Comput}. 24 (1997) 235--265.


\bibitem{Bra2011}\textsc{A. Braun}, On the Gorenstein property for modular invariants. \textit{J. Algebra} 345 (2011) 81--99.

\bibitem{CH1996} \textsc{H.E.A. Campbell \& I. Hughes},  The ring of upper triangular invariants as a module over the Dickson invariants. \textit{Math. Ann.} 306 (1996) 429--443.

\bibitem{CHP1991}\textsc{H.E.A. Campbell, I. Hughes \& R. Pollack}, Rings of invariants and $p$-Sylow subgroups.
  \textit{Canad. Math. Bull.} 34 (1991)  42--47.

\bibitem{CW2011} \textsc{H.E.A. Campbell \& D. Wehlau}, Modular invariant theory. \textit{Encyclopaedia of Mathematical Sciences} 139, Springer-Verlag, Invariant Theory and Algebraic Transformation Groups, 8. Springer-Verlag, Berlin, 2011, xiv+233 pp. ISBN: 978-3-642-17403-2.

\bibitem{Che2014} \textsc{Y. Chen},  On modular invariants of a vector and a covector. \textit{Manuscripta Math.} 144 (2014) 341--348.

\bibitem{DK2002}  \textsc{H. Derksen \& G. Kemper}, Computational invariant theory. \textit{Encyclopaedia of Mathematical Sciences} 130, Springer-Verlag, Invariant Theory and Algebraic Transformation Groups, I, Springer-Verlag, Berlin, 2002, x+268 pp. ISBN: 3-540-43476-3.

\bibitem{Dic1911} \textsc{L. Dickson},  A fundamental system of invariants of the general modular linear group with a solution of the form problem. \textit{Trans. Amer. Math. Soc.} 12 (1911) 75--98.

 \bibitem{Eis1995}  \textsc{D. Eisenbud}, Commutative algebra with a view toward algebraic geometry. \textit{Graduate Texts in Mathematics} 150,  Springer (1995).

 \bibitem{HS2006}  \textsc{C. Huneke \& I. Swanson}, Integral closure of ideals, rings, and modules. \textit{London Mathematical Society Lecture Note Series} 336. Cambridge University Press, Cambridge (2006).


\bibitem{Mui1975}  \textsc{H. Mui}, Modular invariant theory and cohomology algebras of symmetric
groups. \textit{J. Fac. Sci. Univ. Tokyo.} 22 (1975) 319--369.

\bibitem{St1978} \textsc{R. Stanley}, Hilbert Functions of Graded Algebras. \textit{Adv.~Math.} 28 (1978) 57--83.

 \bibitem{Wil1983}  \textsc{C. Wilkerson}, A primer on the Dickson invariants. \textit{Contemp. Math}. 19,  Amer. Math. Soc. (1983).

\end{thebibliography}
\end{document}